\documentclass[11pt]{article}
\usepackage{amssymb, ifthen,amsfonts,amssymb,amsmath,theorem, graphicx}
\usepackage{latexsym}
\usepackage[latin1]{inputenc}
\newcommand{\bTheorem}[1]{\bigskip \begin{thm} \label{T#1}}
\newcommand{\eT}{\end{thm} \bigskip }
\newcommand{\bProposition}[1]{\bigskip \begin{prop} \label{P#1}}
\newcommand{\eP}{\end{prop} \bigskip }
\newcommand{\bLemma}[1]{\bigskip \begin{lem} \label{L#1}}
\newcommand{\eL}{\end{lem} \bigskip }
\newcommand{\bCorollary}[1]{\bigskip \begin{cor} \label{C#1}}
\newcommand{\eC}{\end{cor} \bigskip }
\newcommand{\bFormula}[1]{\begin{equation} \label{#1}}
\newcommand{\eF}{\end{equation}}

\newcommand\eps{\varepsilon}

\newcommand{\HH}{{\mathcal H}}

\newcommand{\Dc}{{\mathcal D}}
\newcommand{\UU}{{\mathcal U}}

\DeclareMathOperator{\cp}{cap}

\newcommand{\ds}{\displaystyle}
\newcommand{\RR}{\mathbb R}
\newcommand{\R}{\RR}
\newcommand{\N}{\mathbb N}

\newcommand{\Z}{\mathbb Z}

\def\M{{\mathcal M}}

\def\Rc{{\mathcal R}}
\def\Lc{{\mathcal L}}

\newcommand{\sm}{\setminus}

\newcommand{\g}{\gamma}
\newcommand{\Om}{\Omega}

\newcommand{\G}{\Gamma}

\newcommand{\lb}{\lambda}
\newcommand{\vps}{\varepsilon}

\newcommand{\vphi}{\varphi}

\newcommand{\ra}{\rightarrow}
\newcommand{\rau}{\rightharpoonup}

\newcommand{\Lra}{\Longrightarrow}
\newcommand{\hr}{\hookrightarrow}

\newcommand{\nif}{{n \rightarrow \infty}}

\newcommand{\sq}{\subseteq}

\newcommand{\MoN}{\M_0 (\RR^N)}
\newcommand{\MoNp}{\M_0^p (\RR^N)}

\newcommand{\NN}{\mathbb N}

\def\bbbq{{\mathchoice {\setbox0=\hbox{$\displaystyle\rm Q$}\hbox{\raise
 0.15\ht0\hbox to0pt{\kern0.4\wd0\vrule height0.8\ht0\hss}\box0}}
 {\setbox0=\hbox{$\textstyle\rm Q$}\hbox{\raise
 0.15\ht0\hbox to0pt{\kern0.4\wd0\vrule height0.8\ht0\hss}\box0}}
 {\setbox0=\hbox{$\scriptstyle\rm Q$}\hbox{\raise
 0.15\ht0\hbox to0pt{\kern0.4\wd0\vrule height0.7\ht0\hss}\box0}}
 {\setbox0=\hbox{$\scriptscriptstyle\rm Q$}\hbox{\raise
 0.15\ht0\hbox to0pt{\kern0.4\wd0\vrule height0.7\ht0\hss}\box0}}}}

\newtheorem{thm}{Theorem}[section]

\newtheorem{exmp}[thm]{Example}
\newtheorem{prop}[thm]{Proposition}
\newtheorem{lem}[thm]{Lemma}
\newtheorem{cor}[thm]{Corollary}

\newcommand{\qed}{\ifmmode$\Box$\else{\unskip\nobreak\hfil
  \penalty50\hskip1em\null\nobreak\hfil$\Box$
  \parfillskip=0pt\finalhyphendemerits=0\endgraf}\fi}

\newtheorem{demth}{Proof}
  
  \newenvironment{proof}{\begin{demth}\rm}{\qed\end{demth}}

\title{On the characterization of the compact embedding of Sobolev spaces}

\author {Dorin Bucur\thanks{\scriptsize\ Laboratoire de Math\'ematiques (LAMA), Universit\'e de Savoie, Campus Scientifique - 73376 Le-Bourget-Du-Lac, FRANCE 
\texttt{dorin.bucur@univ-savoie.fr}}\qquad\qquad Giuseppe Buttazzo\thanks{\scriptsize\ Dipartimento di Matematica, Universit\`a di Pisa, Largo B. Pontecorvo 5 - 56127 Pisa, ITALY
\texttt{buttazzo@dm.unipi.it}}
}
\date{}

\begin{document}

\maketitle

\bigskip

\begin{abstract}
For every positive regular Borel measure, possibly infinite valued, vanishing on all sets of $p$-capacity zero, we characterize the compactness of the embedding $W^{1,p}(\RR^N)\cap L^p (\RR^N,\mu)\hr L^q(\RR^N)$ in terms of the qualitative behavior of some characteristic PDE. This question is related to the well posedness of a class of geometric inequalities involving the torsional rigidity and the spectrum of the Dirichlet Laplacian introduced by Polya and Szeg\"o \cite{posz51} in 1951. In particular, we prove that finite torsional rigidity of an arbitrary domain (possibly with infinite measure), implies the compactness of the resolvent of the Laplacian.
\end{abstract}



\section{Introduction}\label{secintr}

The study of the compact embeddings of Sobolev spaces $W^{1,p}_0(\Om)$ into $L^p(\Om)$ has a long history; starting from the simplest case when $\Om$ is bounded, in which the compactness always occurs, several generalizations have been found (see for instance \cite{adfo}). For $p=2$, this is related (actually equivalent) to the compactness of the resolvent operator $\Rc_\Om:L^2(\Om)\to L^2(\Om)$ which associates to every function $f\in L^2(\Om)$ the solution of the elliptic PDE
$$\left\{\begin{array}{ll}
-\Delta u+u=f\quad\hbox{in }\Om\\
u\in H^1_0(\Om),
\end{array}\right.$$
then providing the existence of a discrete spectrum made by a nondecreasing sequence of eigenvalues.

Analogously, the same kind of problems arise for the Schr\"odinger operator
$$\left\{\begin{array}{ll}
-\Delta u+u+V(x)u=f\quad\hbox{in }\R^N\\
u\in H^1(\R^N);
\end{array}\right.$$
the compactness of its resolvent operator $\Rc_V:L^2(\R^N)\to L^2(\R^N)$ is again an important issue for its spectral analysis.

We unify the two topics considering the so-called {\it capacitary measures} and the related Sobolev spaces $W^{1,p}_\mu$ (see Section \ref{secsobo} for the precise definitions); when $\mu=\infty_{\R^N\sm\Om}$ we recover the usual Sobolev space $W^{1,p}_0(\Om)$, while $\mu=V(x)\,dx$ provides the natural space for the Schr\"odinger operator.

The main results of the paper deal with some characterizations for the compact embeddings
$$W^{1,p}_\mu:=W^{1,p}(\RR^N)\cap L^p_\mu\hr L^p(\R^N)\quad\hbox{and}\quad W^{1,p}_\mu\hr L^1(\R^N)$$
in terms of the qualitative behavior of the formal solution (see the precise definition in Section \ref{secsobo}) of the equation
$$-\Delta_p w+w^{p-2}w+\mu w^{p-2} w=1.$$
Precisely, we prove that the (inclusion and) compactness $W^{1,p}_\mu\hr L^1(\R^N)$ is equivalent to  $\int_{\R^N}w\,dx < +\infty$ and the compactness $W^{1,p}_\mu\hr L^p(\R^N)$ is equivalent to the uniform vanishing at infinity of $w$.

Of course, as soon as the compact embedding $W^{1,p}(\RR^N)\cap L^p_\mu\hr L^q(\RR^N)$ holds for $q=1$ or $q=p$, the embedding $W^{1,p}(\RR^N)\cap L^p_\mu\hr L^r(\RR^N)$ is also compact for every $q \le r < p^* $, where $p^*=Np/(N-p)$, by a standard argument based on H\"older inequality and completeness of $L^r(\R^N)$ (see for instance \cite[Lemma 6.7]{adfo}).

Clearly, if  the torsional rigidity of the measure $\mu$ is finite (take $p=2$) then  by the maximum principle $\int_{\R^N}w\,dx < +\infty$ so that the embedding $H^1_\mu \hr L^q(\R^N)$ holds for every $1<q<2^*$. The torisional rigidity of $\mu$ is defined by 
$$P(\mu)=\int_{\R^N} u dx,\;\;\mbox{where}\;-\Delta u +\mu u=1\;\mbox{in}\;(H^1_\mu)', \; \mbox{and}\;u \in H^1_\mu.$$
The question of analysing the torsional rigidity and the torsion function in relationship with the geometric domain and the spectrum of the Dirichlet Laplacian was already addressed in \cite{vdBC} and \cite{BvdBC}. Precisley, in these papers the authors are interested to situations when the torsion function belongs to $L^\infty (\R^N)$ and the torsional rigidity is finite. As well, sufficient conditions expressed in terms of the distance function to the boundary of the domains give information about summability of $u$.

For the simplicity of the exposition we prove all results for $p=2$, which is more rich than the nonlinear framework. In the last section, we briefly consider the general case $1<p<+\infty$ for which we point out the main differences with respect to the Hilbertian case.

\section{The Sobolev space $H^1_\mu$}\label{secsobo}

We will use in the following the notion of {\it capacity} of a subset $E$ of $\R^N$, defined by
$$\cp(E)=\inf\Big\{\int_{\R^N}|\nabla u|^2+u^2\,dx\ :\ u\in\UU_E\Big\}\,,$$
where $\UU_E$ is the set of all functions $u$ of the Sobolev space $H^1(\R^N)$ such that $u\ge1$ almost everywhere in a neighbourhood of $E$. Below we summarize the main properties of the capacity and the related convergences. For further details we refer to \cite{bubu05} or to \cite{hepi05}.

If a property $P(x)$ holds for all $x\in E$ except for the elements of a set $Z\subset E$ with $\cp(Z)=0$, then we say that $P(x)$ holds {\it quasi-everywhere} (shortly {\it q.e.}) on $E$. The expression {\it almost everywhere} (shortly {\it a.e.}) refers, as usual, to the Lebesgue measure.

A subset $\Omega$ of $\R^N$ is said to be {\it quasi-open} if for every $\eps>0$ there exists an open subset $\Omega_\eps$ of $\R^N$, such that $\cp(\Omega_\eps{\scriptstyle\Delta}\Omega)<\eps$, where ${\scriptstyle\Delta}$ denotes the symmetric difference of sets. Equivalently, a quasi-open set $\Omega$ can be seen as the set $\{u>0\}$ for some function $u$ belonging to the Sobolev space $H^1(\R^N)$. Note that a Sobolev function is only defined quasi-everywhere, so that a quasi-open set $\Omega$ does not change if modified by a set of capacity zero.

A function $f:\R^N\to\R$ is said to be {\it quasi-continuous} (respectively {\it quasi-lower semicontinuous} if for every $\eps>0$ there exists a continuous (respectively lower semicontinuous) function $f_\eps:\R^N\to\R$ such that $\cp(\{f\ne f_\eps\})<\eps$. It is well known (see, e.g., Ziemer \cite{ziemer}) that every function $u$ of the Sobolev space $H^1(\R^N)$ has a quasi-continuous representative, which is uniquely defined up to a set of capacity zero. We shall always identify the function $u$ with its quasi-continuous representative, so that a pointwise condition can be imposed on $u(x)$ for quasi-every $x\in\R^N$. Notice that with this convention we have
$$\cp(E)=\min\Big\{\int_{\R^N}|\nabla u|^2+|u|^2\,dx\ :\ u\in H^1(\R^N),\ 
u\ge1\hbox{ q.e. on }E\Big\}.$$
For every quasi-open set $\Omega\subset\R^N$ we denote by $H^1_0(\Omega)$ the space of all functions $u\in H^1(\R^N)$ such that $u=0$ {\it q.e.} on $\R^N\setminus\Omega$, endowed with the Hilbert space structure inherited from $H^1(\R^N)$. In this way $H^1_0(\Omega)$ is a closed subspace of $H^1(\R^N)$. If $\Omega$ is open, then the definition above of $H^1_0(\Omega)$ is equivalent to the usual one (see \cite{ah96}). If $\Omega$ is bounded the linear operator $-\Delta$ on $H^1_0(\Omega)$ has a compact resolvent, hence a discrete spectrum, denoted by $\lambda_1(\Omega)\le\lambda_2(\Omega)\le\lambda_3(\Omega)\le\cdots$; for general $\Omega$ this is not true and the question is related to the compact embedding of $H^1_0(\Omega)$ into $L^2(\Omega)$ which will be considered in the next section.

More generally, we can consider the Sobolev spaces $H^1_\mu$ made with respect to the so-called {\it capacitary measures}; precisely, we consider nonnegative regular Borel measures $\mu$ on $\R^N$, possibly $+\infty$ valued, that vanish on all sets of capacity zero. The family of these measures is denoted by $\MoN$. We stress the fact that the measures $\mu$ above do not need to be finite, and may take the value $+\infty$ even on large parts of $\R^N$.

\begin{exmp}{\rm
If $N-2<\alpha\le N$ the $\alpha$-dimensional Hausdorff measure $\HH^\alpha$ is a capacitary measure (and consequently every $\mu$ absolutely continuous with respect to $\HH^\alpha$ as well). In fact all Borel sets with capacity zero have a Hausdorff dimension which is less than or equal to $N-2$.
}\end{exmp}

\begin{exmp}{\rm
Another example of capacitary measure is, for every $S\subset\R^N$, the measure $\infty_S$ defined by
\begin{equation}\label{e421}
\infty_S(B)= \left\{
\begin{array}{ll}
0&\mbox{if }\cp(B\cap S)=0,\\
+\infty&\mbox{otherwise}.
\end{array}\right.
\end{equation}
}\end{exmp}

The norm
$$\|u\|^2_{1,\mu}=\int_{\R^N}\big(|\nabla u|^2+|u^2|\big)\,dx+\int_{\R^N} |u|^2\,d\mu$$
makes
$$H^1_\mu=\big\{u\in H^1(\R^N)\ :\ \|u\|_{1,\mu}<+\infty\big\}$$
a Hilbert space, and for every $f\in L^2(\R^N)$ (or more generally for $f\in(H^1_\mu)'$) we may consider the elliptic PDE
\begin{equation}\label{pde}
-\Delta u+u+\mu u=f,\qquad u\in H^1_\mu
\end{equation}
whose precise sense has to be given in the weak form
$$\int_{\R^N}\big(\nabla u\nabla\phi+u\phi)\,dx+\int_{\R^N} u\phi\,d\mu=\int_{\R^N}f\phi \,dx \qquad\forall\phi\in H^1_\mu.$$
Notice that, since the Sobolev functions $u$ are defined quasi-everywhere and the capacitary measures $\mu$ vanish on all sets with capacity zero, the products $u\mu$ are well defined. In particular, if $\mu=\infty_S$ we have
$$H^1_\mu=\big\{u\in H^1(\R^N)\ :\ u=0\hbox{ q.e. on }S\big\}.$$

By standard Lax-Milgram methods, for every $f\in L^2(\R^N)$ equation \eqref{pde} has a unique solution, that we denote by $\Rc_\mu(f)$; in this way we may define the resolvent operator $\Rc_\mu:L^2(\R^N)\to L^2(\R^N)$ whose compactness will be discussed in the next section.

For every $\mu\in\MoN$ and every quasi-open set $\Omega$ we define the Dirichlet restriction
$$\mu\lceil\Omega=\mu+\infty_{\R^N\setminus\Omega}$$
which takes the value $+\infty$ outside $\Omega$; in other words, solving the PDE \eqref{pde} with $\mu\lceil\Omega$ means that we are considering the same PDE but with Dirichlet condition $u\in H^1_0(\Omega)\cap H^1_\mu $. The classical restriction of measures is denoted by $\mu\lfloor \Om$.

The space $\MoN$ of all capacitary measures can be endowed with an interesting convergence structure, called $\g$-convergence: we say that $\mu_n\to\mu$ in the $\g$ convergence if for every ball $B$
$$\Rc_{\mu_n\lceil B}\to\Rc_{\mu\lceil B}\mbox{ in }\Lc\big(L^2(\R^N)\big).$$
It is well known (see for instance \cite{bubu05}), that the $\g$-convergence is equivalent to any of the assertions below
\begin{itemize}
\item
$\Rc_{\mu_n\lceil B}(f)\to\Rc_{\mu\lceil B}(f)$ weakly in $L^2(\R^N)$, for all balls $B$ and for all $f\in L^2(\R^N)$
\item
$\Rc_{\mu_n\lceil B}(1)\to\Rc_{\mu\lceil B}(1)$ weakly in $L^2(\R^N)$, for all $R>0$
\end{itemize}

The $\g$-convergence is metrizable and the family of measures $\MoN$ is compact for the $\g$-convergence. Clearly, the spectrum of the operator in \eqref{pde} is $\g$-continuous on the families $\{\mu\lceil B_R\ :\ \mu\in\MoN\}$, but in general is not $\g$-continuous on $\MoN$. Moreover, the family of measures of the form $\infty_S$ with $S$ smooth and compact (that we often identify with the domain $\R^N\setminus S$) is $\g$-dense in $\MoN$ as well as the family of measures of the form $a(x)\,dx$ with $a(x)$ smooth.

In the following we shall use the function $w_\mu$ that {\it formally} solve the PDE
$$-\Delta u+u+\mu u=1.$$
Since in general the constant $1$ does not belong to $(H^1_\mu)'$ we define $w_\mu$ as
$$w_\mu=\lim_{R\to+\infty}\Rc_{\mu\lceil B_R}(1).$$
By the maximum principle the limit above exists since the solutions $\Rc_{\mu\lceil B_R}(1)$ are monotonically increasing with $R$; moreover, it is easy to see that $0\le\Rc_{\mu\lceil B_R}(1)\le1$, so that $0\le w_\mu\le1$.

In this way, there is a classical extension of the operator $\Rc_\mu$ on $L^\infty (\R^N)$, defined by
$$\Rc_\mu(f)=\sup_R \Rc_{\mu\lceil B_R}(f^+)-\sup_R \Rc_{\mu\lceil B_R}(f^-),$$ 
which is linear and continuous (see for instance \cite{bida06}).

For every measure $\mu$ we denote by $\lb_1(\mu)$ the spectral abscissa of the Laplacian associated to the measure $\mu$ by
$$\lb_1(\mu)=\inf_{u\in H^1_\mu,\ u\ne0}\frac{\int_{\R^N}|\nabla u|^2\,dx
+\int_{\R^N}u^2\,dx+\int_{\R^N}u^2\,d\mu}{\int_{\R^N}u^2\,dx}.$$
Clearly if $H^1_\mu$ is compactly embedded in $L^2(\R^N)$, the spectral abscissa is the first eigenvalue of the operator in \eqref{pde}. By an abuse of notation, we still denote it $\lb_1(\mu)$, even if the compact embedding does not hold. 

Throughout the paper we consider a given function $\theta\in C^\infty_c(B_2)$ such that $0\le\theta\le1$ and $\theta=1$ on $B_1$. For every $R>0$, we set $\theta _R(x)=\theta (\frac{x}{R})$. We shall often use the fact that for every function $u\in H^1_\mu$ we have $u\theta_R\to u$ strongly in $H^1_\mu$ as $R\to+\infty$, and the estimate
$$\begin{array}{ll}
\ds\int_{\R^N}|\nabla(u\theta_R)|^2\,dx&\ds\le2\int_{\R^N}|\nabla u|^2\,dx+
2\int_{\R^N}u^2|\nabla\theta_R|^2\,dx\\
&\ds\le2\int_{\R^N}|\nabla u|^2\,dx+
\frac{2|\nabla\theta|_\infty^2}{R^2}\int_{\R^N}u^2\,dx.
\end{array}$$

\section{Characterization of the compactness in the linear frame}\label{secchar}

Here are the main results of the paper. 
\begin{thm}\label{t01}
Let $\mu\in\MoN$. Then the embedding $H^1_\mu\hr L^2(\R^N)$ is compact if and only if 
$$w_\mu\cdot1_{B_R^c}\to0\mbox{ in $L^\infty(\RR^N)$ as }R\to+\infty.$$
\end{thm}

\begin{thm}\label{t02}
Let $\mu \in \MoN$. The following assertions are equivalent.
\begin{enumerate}
\item $w_\mu \in L^1(\R^N)$.
\item $H^1_\mu\subset L^1(\R^N)$ with continuous injection.
\end{enumerate}
Moreover, if one of the two assertions above holds, then the embedding $H^1_\mu\hr L^1(\R^N)$ is compact.
\end{thm}

For the sake of completeness, in Proposition \ref{t01} below we collect several characterizations of the compact embedding in $L^2(\R^N)$ in the case of capacitray measures. Some of them are only direct extensions to capacitary measures of well known characterizations holding for domains or positive potentials absolutely continuous with respect to the Lebesgue measure, of the form $V(x) dx$. In the case of open sets, condition 4) is a consequence of \cite[Theorem 3.1]{bida06} while condition 8) is an extension to measures of the characterization of the compactness in \cite[Theorem 6.19]{adfo}.

\begin{prop}\label{p01}
Let $\mu\in\MoN$. The following assertions are equivalent.
\begin{enumerate}
\item $H^1_\mu\hr L^2(\R^N)$ is compact;
\item $\Rc_\mu$ is compact;
\item $\Rc_{\mu\lceil B_R}\to\Rc_\mu$ in $\Lc(L^2(\R^N))$ as $R\to+\infty$;
\item $\Rc_{\mu\lceil B_R}\to\Rc_\mu$ in $\Lc(L^\infty(\R^N))$ as $R\to+\infty$;
\item $\lb_1(\mu\lceil B_R^c)\to+\infty$ as $R\to+\infty$;
\item for every $h>0$, $\lb_1(\mu\lceil B_h(x))\to+\infty$ as $\|x\|\to+\infty$;
\item there exists $h>0$, $\lb_1(\mu\lceil B_h(x))\to+\infty$ as $\|x\|\to+\infty$;
\item there exists $h>0$, such that for every $\vps>0$ there exists $r\ge0$ such that
$$\inf_{u\in H^1_{\mu\lfloor H}(\R^N),\ u\ne0}\frac{\int_H|\nabla u|^2\,dx
+\int_H u^2\,d\mu}{\int_Hu^2\,dx}\ge\frac{1}{\vps},$$
for every $N$-cube $H=[x,x+h)$, with $\|x\|\ge r$.
\end{enumerate}
\end{prop}

In the sequel we often use the notation $[x,y)$ for the set $\Pi_{i=1}^N[x_i,y_i)$. We now give some examples, some of them classical, in order to highlight the various conditions in Theorems \ref{t01}, \ref{t02} and Proposition \ref{p01}.

\begin{exmp}\label{ex00}{\rm
Let $\Om\sq\R^N$ be an open set such that for some $h>0$ $|\Om\cap B_h(x)|\to0$ as $\|x\|\to+\infty$. It is well known that in this case the embedding $H^1_0(\Om)\hr L^2(\Om)$ is compact. Indeed, the Faber-Krahn inequality gives that 
$$\lb_1(\Om\cap B_h(x))\ge \lb_1([\Om\cap B_h(x)]^*),$$
where $[\Om\cap B_h(x)]^*$ is the ball of the same volume as $\Om\cap B_h(x)$. Since $|\Om\cap B_h(x)|\to0$, then $\lb_1([\Om\cap B_h(x)]^*)\to+\infty$, so that condition 7) in Proposition \ref{p01} is satisfied.
}\end{exmp}

\begin{exmp}\label{ex01}{\rm
Let $\Om\sq\RR^N$ be an open set such that $|\Om|<+\infty$. Then $H^1_0(\Om)\hr L^2(\Om)$ is compact. This is an immediate consequence of Example \ref{ex00}
}\end{exmp}

\begin{exmp}{\rm
There exist situations when the conditions of Example \ref{ex01} is only sufficient and not necessary for compactness. Simply consider in $\RR^2$
$$\Om=(0,+\infty)\times(0,1)\sm\bigcup_{n\in\NN}\{x_n\}\times(0,1),$$
where
$$x_n=\log(1+n).$$
Clearly, $|\Om|=+\infty$ and also $|\Om\cap B_1(x_n,0)|\not\to0$, but condition 7) of Proposition \ref{p01} is satisfied since
$$\lb_1\big((x_n,x_{n+1})\times(0,1)\big)\to+\infty.$$
}\end{exmp}

\begin{exmp}{\rm
We give here an example showing that higher order sumability of $w_\mu$ is not related to the compact embedding of $H^1_\mu $ in $L^2(\R^N)$. Consider as in the previous example
$$\Om=(0,+\infty)\times(0,1)\sm\bigcup_{n\in\NN}\{x_n\}\times(0,1),$$
where $(x_n)_n$ is an increasing sequence of positive numbers such that $x_{n+1}-x_n\to0$. This readily gives the compact embedding of $H^1_0(\Om)$ in $L^2(\R^N)$. Clearly, we can tune the $x_n$ such that for some $\alpha>0$ $\int w_\mu^\alpha\,dx=+\infty$. In particular, if $\alpha=1$ there is no compact embedding in $L^1(\R^N)$.
}\end{exmp}

\begin{exmp}{\rm
Obviously, the measure $\mu$ has a decisive influence on the compactness. It is well known that if $\mu=V(x)\,dx$, with $V:\R^N\to[0, +\infty)$, measurable and $V(x)\to+\infty$ as $\|x\|\to+\infty$, then $H^1_\mu\hr L^2(\R^N)$ is compact as a consequence of condition 7) in Proposition \ref{p01}. In other words, the Schr\"odinger operator $-\Delta u+u+V(x)u$ has a compact resolvent. 
}\end{exmp}

\begin{exmp}{\rm
In order to get the compactness embedding $H^1_\mu\hr L^2(\R^N)$ for measures of the form $\mu=V(x)\,dx$, it is not necessary to require that $V(x)\to+\infty$ as $\|x\|\to+\infty$. Indeed, in $\R^2$ one can consider for instance $V(x_1,x_2)=|x_1|^\alpha|x_2|^\beta$, for some $\alpha,\beta>0$. In this case, one can prove easily that condition 7) in Proposition \ref{p01} is still satisfied, by analyzing the $\g$-convergence of the measures $Vdx\lceil B_h(x_n)$.
}\end{exmp}

\begin{exmp}{\rm
Let $\mu=V(x) dx$, with $V:\R^N\to(0,+\infty)$, measurable. If there exists $\beta \in (0,1]$ such that $\frac{1}{V^\beta}\in L^1 (\R^N)$, then 
$H^1_\mu$ is compactly embedded in $L^1(\R^N)$ (hence in $L^2(\R^N)$). Indeed, we have
$$\int_{\R^N}w_{\mu\lceil B_R}\,dx
=\int_{\R^N}V^\alpha w_{\mu\lceil B_R}\frac{1}{V^\alpha}\,dx
\le\Big(\int_{\R^N}[V^\alpha w_{\mu\lceil B_R}]^p\,dx\Big)^{\frac{1}{p}}
\Big(\int_{\R^N}\frac{1}{V^{\alpha q}}\,dx\Big)^{\frac{1}{q}},$$
where $p>1$ and $\frac{1}{p}+\frac{1}{q}=1$.

Choosing $\alpha=\frac{1}{p}$ and noticing that for $p\ge2$ we have $\alpha q \in(0,1]$, we fix $p$ such that $\alpha q=\beta$, so $\int_{\R^N}\frac{1}{V^{\alpha q}}\,dx<+\infty$. For the first term, we use the fact that $w _{\mu\lceil B_R} \le 1$ and $p \ge 2$, so that
$$\int_{\R^N}w_{\mu\lceil B_R}\,dx
\le\Big(\int_{\R^N}V w_{\mu\lceil B_R}^2\,dx
\Big)^{\frac{1}{p}}|V^\beta|_{L^1(\R^N)}^{\frac{1}{q}}.$$
Using the equation satisfied by $w_{\mu\lceil B_R}$, we obtain
$$\int_{\R^N}V w_{\mu\lceil B_R}^2\le\int_{\R^N}w_{\mu\lceil B_R}\,dx$$
thus we get
$$\Big(\int_{\R^N}w_{\mu\lceil B_R}\,dx\Big)^{1-\frac{1}{p}}
\le|V^\beta|_{L^1(\R^N)}^{\frac{1}{q}}.$$
Passing to the limit as $R\to\infty$, we conclude that $w_\mu\in L^1(\R^N)$ so that Theorem \ref{t02} applies.
}\end{exmp}

\section{Further remarks and applications}\label{secfurt}

The main motivation of the paper originates from a conjecture of Polya and Szeg\"o \cite{posz51} which states that among all simply connected membranes $\Om\sq \R^2$, the minimum of the product $P(\Om)\lb_1^2 (\Om)$ is attained on balls. Here $P(\Om)$ stands for the torsional rigidity and $\lb_1(\Om)$ for the first Dirichlet eigenvalue of the Laplacian on $\Om$. By definition
$$P(\Om)=\int_\Om u\,dx,$$
where $u\in H^1_0(\Om)$ solves in $\Dc'(\Om)$
$$-\Delta u=1.$$
The conjecture was proved in 1978 by Kohler-Jobin in \cite{kj78} and extended to inhomogeneous membranes in \cite{kj78bis} by a rather sophisticated ``dearrangement'' procedure. Naturally, we can reframe the problem as
\begin{equation}\label{eq04}
\min\{\lb_1(\Om)\ :\ \Om\sq\R^N,\ P(\Om)=c\}
\end{equation}
and obtain a sort of ``isoperimetric'' problem, where the usual constraint on $\Om$ set in terms of area or perimeter is replaced by a constraint on torsional rigidity. A natural generalization of this problem is the following: for $k\in\N$ solve
\begin{equation}\label{eq05}
\min\{\lb_k(\Om)\ :\ \Om\sq\R^N,\ P(\Om)=c\}.
\end{equation}
More general functionals of the form $F(\lb_1(\Om),\dots,\lb_k(\Om))$ can also be considered.

Notice, that in usual isoperimetric inequalities, the constraint on $\Om$ is of the form $|\Om|=c$ or $\HH ^{N-1}(\partial \Om)=c$. Both of them imply that the resolvent operator of the Dirichlet Laplacian is compact so that the spectrum is well defined. A priori, it is not obvious that finite torsional rigidity alone would imply the same property. Of course, one can add an artificial constraint by setting that $\Om$ is bounded. Nevertheless, from a variational point, this amounts to restrict the class of admissible domains to a non-closed one.

It is not difficult to observe that for $k=2$ the solution of \eqref{eq05} consists on two disjoint and equal balls. For $k=3$ and $k=4$ numerical computations based on a genetic algorithm (see \cite{durus,bdo08}) lead to the intuition that the solution consists on $3$, respectively $4$, equal and disjoint balls. Although it is clear that for classical isoperimetric inequality for eigenvalues, the minimzer for $\lb_3$ and $\lb_4$ is not the union of $3$ or $4$ equal balls (in 2D), the same arguments are not valid for problem \eqref{eq05}. Starting from the numerical computations above, we are led to the following problems.

\medskip
\noindent {\bf Problem 1.} Prove or disprove that for $k=3,4$ in $\R^2$ the solution consists of $k$ equal and disjoint balls.

\medskip
\noindent {\bf Problem 2.} Is a similar assertion true for every $k$ and every dimension of the space?

\medskip
The main consequence of Theorem \ref{t02} is that problem \eqref{eq05} is well posed in the family of all open sets (possibly unbounded or of infinite measure of $\R^N$) with finite torsional rigidity. This is a consequence of the fact that if $P(\Om) < + \infty$, then $w_\Om\in L^1(\Om)$. By Theorem \ref{t02} the resolvent operator is compact, so the spectrum of the Dirichlet Laplacian consists only on eigenvalues, thus \eqref{eq05} is well posed.

Moreover, the problem has a natural extension on the family of measures $\MoN$, and we know that the family of domains is ``dense'' in the sense of $\g$-convergence in $\MoN$ (see section \ref{secsobo}). If classical isoperimetric inequalities are hardly well written on measures, since the perimeter or the area of a measure $\mu$ has no mechanical meaning, the torsional rigidity of a measure is well defined.

\section{Proofs of the main results}\label{secproo}

\begin{proof}[of Theorem \ref{t01}]
\medskip
\noindent {\bf Necessity.} Assume by contradiction that there exists $\delta>0$ and a sequence of points $x_n$ with $\|x_n\|\to+\infty$ such that for every $r>0$
$$\big|B_r(x_n)\cap\{x\ :\ w_\mu(x)\ge\delta\}\big|>0.$$
In view of the definition of $w_\mu$, for every $n\in\N$ there exists $R_n$ such that for every $r>0$
$$\big|B_r(x_n)\cap\{x\ :\ w_{\mu\lceil B_{R_n}}(x)\ge\delta/2\}\big|>0.$$
We introduce the functions
$$\vphi_n(x)=w_{\mu\lceil B_{R_n}}(x)\theta(x+x_n),$$
and we prove that $\vphi_n$ is bounded in $H^1_\mu$, converges to $0$ weakly in $H^1_\mu$ but does not converge strongly in $L^2(\R^N)$.

In order to bound the $H^1_\mu$-norm, we take $\vphi_n$ as test function in the equation satisfied by $w_{\mu\lceil B_{R_n}}$. So
$$\int_{\R^N}\nabla w_{\mu\lceil B_{R_n}}\nabla\vphi_n\,dx
+\int_{\R^N}w_{\mu\lceil B_{R_n}}\vphi_n\,dx
+\int_{\R^N}w_{\mu\lceil B_{R_n}}\vphi_n d\mu
=\int_{\R^N}\vphi_n\,dx.$$
Simple computations lead to
$$\begin{array}{ll}
&\ds\int_{\R^N}|\nabla w_{\mu\lceil B_{R_n}}|^2\theta(\cdot+x_n)\,dx
+\int_{\R^N}\nabla w_{\mu\lceil B_{R_n}}\nabla\theta(\cdot+x_n)w_{\mu\lceil B_{R_n}}\,dx
+\int_{\R^N}w_{\mu\lceil B_{R_n}}^2\theta(\cdot+x_n)\,dx\\
&\ds\hskip3truecm+\int_{\R^N}w_{\mu\lceil B_{R_n}}^2\theta(\cdot+x_n)\,d\mu
=\int_{\R^N}w_{\mu\lceil B_{R_n}}\theta(\cdot+x_n)\,dx.
\end{array}$$
Since $0\le w_{\mu\lceil B_{R_n}}\le1$, $0\le\theta\le1$, and $\theta$ has its support in $B_2$, we have in the right hand side
$$\int_{\R^N}w_{\mu\lceil B_{R_n}}\theta(\cdot+x_n)\,dx\le|B_2|.$$
Performing an integration by parts on the second term in left hand side, we get
$$\int_{\R^N}\nabla w_{\mu\lceil B_{R_n}}\nabla\theta(\cdot+x_n)w_{\mu\lceil B_{R_n}}\,dx
=-\frac{1}{2}\int_{\R^N}w_{\mu\lceil B_{R_n}}^2\Delta\theta(\cdot+x_n)\,dx,$$
so that
$$\Big|\int_{\R^N}\nabla w_{\mu\lceil B_{R_n}}\nabla\theta(\cdot+x_n)w_{\mu\lceil B_{R_n}}\,dx\Big|
\le\frac{1}{2}|B_2|\|\Delta\theta\|_\infty.$$
Finally, we get that
$$\begin{array}{ll}
&\ds\int_{\R^N}|\nabla w_{\mu\lceil B_{R_n}}|^2\theta(\cdot+x_n)\,dx
+\int_{\R^N}w_{\mu\lceil B_{R_n}}^2\theta(\cdot+x_n)\,dx\\
&\ds+\int_{\R^N}w_{\mu\lceil B_{R_n}}^2\theta(\cdot+x_n)\,d\mu
\le|B_2|\Big(1+\frac{1}{2}\|\Delta\theta\|_\infty\Big).
\end{array}$$
Using again that $0\le\theta\le1$ we obtain
$$\|\vphi_n\|^2_{H^1_\mu}\le2|B_2|\Big(1+\frac{1}{2}\|\Delta\theta\|_\infty
+\|\nabla\theta\|_\infty^2\Big),$$
so $(\vphi_n)_n$ is bounded in $H^1_\mu$.

We notice that $\vphi_n\rau0$ weakly in $H^1_\mu$ since the support of $\vphi_n$ lies in a ball of radius $2$ centered at the point $x_n$ which goes to infinity. In order to prove that $\vphi_n$ does not converge strongly in $L^2$ to $0$, it is enough to show that its $L^1$-norm
does not converge to $0$, since
$$\int_{\R^N}\vphi^2_n\,dx=\int_{B_2(x_n)}\vphi^2_n\,dx
\ge\frac{1}{|B_2|}\Big(\int_{B_2(x_n)}\vphi_n\,dx\Big)^2
=\frac{1}{|B_2|}\Big(\int_{\R^N}\vphi_n\,dx\Big)^2.$$
We have
$$\int_{B_2(x_n)}\vphi_n\,dx\ge\int_{B_1(x_n)}w_{\mu\lceil B_{R_n}}\,dx.$$
Since
$$-\Delta w_{\mu\lceil B_{R_n}}\le1\mbox{ in }\Dc'(\R^N),$$
we have
$$\Delta\Big(w_{\mu\lceil B_{R_n}}(x)+\frac{|x-x_n|^2}{2N}\Big)\ge0\mbox{ in }\Dc'(\R^N).$$
Consequently, since there are Lebesgue points of the set $\{x\ :\ w_{\mu\lceil B_{R_n}}(x)\ge\frac{\delta}{2}\}$ in any neighborhood of $x_n$, for every $r>0$ we get
$$|B_r(x_n)|\frac{\delta}{2}\le
\int_{B_r(x_n)}\Big(w_{\mu\lceil B_{R_n}}(x)+\frac{|x-x_n|^2}{2N}\Big)\,dx
=\int_{B_r(x_n)}w_{\mu\lceil B_{R_n}}(x)\,dx+c_Nr^{N+2}.$$
Therefore, there exists $r$ small enough depending only on $\delta$, such that
\begin{equation}\label{eq03}
\int_{B_r(x_n)}w_{\mu\lceil B_{R_n}}(x)\,dx\ge|B_r(x_n)|\frac{\delta}{4},
\end{equation}
hence $\vphi_n$ does not converge strongly in $L^2$ to $0$.

\bigskip
\noindent {\bf Sufficiency.} We start with the following 

\begin{lem}\label{l01}
Let $\mu\in\MoN$ be such that for some $0<\vps<1$ we have that $w_\mu\le\vps$. Then
$$\lb_1(\mu)\ge\frac{1}{\vps}.$$
\end{lem}

\begin{proof}[of Lemma \ref{l01}]
From the density of $\{u\theta_R\ :\ u\in H^1_\mu,\ R>0\}$ in $H^1_\mu$, it is enough to prove the assertion for a measure $\mu\lceil B_R$. Moreover, using the density for the $\g$-convergence of bounded open sets in the family measures with bounded regular set, it is enough to prove the assertion of the theorem only for bounded open sets.

Let $\Om$ be a bounded open set and let $u_1$ denote a nonzero first eigenfunction for the operator $-\Delta u+u$. Then $u_1\in L^\infty(\Om)\cap H^1_0(\Om)$, and we have
$$-\Delta u_1+u_1\le\lb_1(\Om)\|u_1\|_\infty,$$
so by monotonicity we get
$$0\le u_1\le w_{\Om}\lb_1(\Om)\|u_1\|_\infty,$$
thus passing to the supremum on the left hand side and using the hypothesis $w_\Om\le\vps$ we obtain
$$1\le\vps\lb_1(\Om),$$
which gives the conclusion.
\end{proof}

Coming back to to the proof of the sufficiency part, we can use the lemma above and the hypothesis $w_\mu\cdot1_{B_R^c}\to0$ in $L^\infty(\RR^N)$ as $R\to+\infty$. So, for every $\vps>0$ there exists $R_\vps$ such that for $R\ge R_\vps$
$$w_\mu\cdot1_{B_R^c}\le\vps\qquad\hbox{a.e.}$$
By monotonicity, we get that 
$$w_{\mu\lceil B_R^c}\le w_\mu\cdot1_{B_R^c}\le\vps\qquad\hbox{a.e.}$$
hence Lemma \ref{l01} gives that $\lb_1(\mu\lceil B_R^c)\ge\frac{1}{\vps}$. Making $\vps\to0$ we get that $\lb_1(\mu\lceil B_R^c)\to+\infty$, as $R\to+\infty$.

Let $\{u_n\}_n\sq H^1_\mu$ be a bounded sequence and assume $u_n\to0$ weakly in $H^1_\mu$, with $\|u_n\|_{H^1_\mu}\le M$. Then for every $R>0$ we have that $u_n\theta_R\to0$ strongly in $L^2(\R^N)$. But $u_n(1-\theta_R)\in H^1_{(\mu\lceil B_R^c)}$ so that it can be taken as test function for $\lb_1(\mu\lceil B_R^c)$. Hence
$$\begin{array}{ll}
&\lb_1(\mu\lceil B_R^c)\le\frac{\ds\int_{\R^N}\big|\nabla\big(u_n(1-\theta_R)\big)\big|^2\,dx
+\int_{\R^N}u_n^2(1-\theta_R)^2\,dx
+\int_{\R^N}u_n^2(1-\theta_R)^2\,d\mu}{\ds\int_{\R^N}u_n^2(1-\theta_R)^2\,dx}\\
&\hskip1.9truecm\le\frac{\ds2M\big(1+R^{-2}|\nabla\theta|_\infty^2\big)}
{\ds\int_{\R^N}u_n^2(1-\theta_R)^2\,dx},
\end{array}$$
so that
$$\int_{\R^N}u_n^2(1-\theta_R)^2\,dx
\le\frac{2M\big(1+R^{-2}|\nabla\theta|_\infty^2\big)}{\lb_1(\mu\lceil B_R^c)}.$$
By a standard argument we get that $u_n\to0$ strongly in $L^2(\R^N)$.
\end{proof}

\begin{proof}[of Theorem \ref{t02}]
{\bf 1) $\Rightarrow$ 2)} Assume $w_\mu\in L^1(\R^N)$. By the definition of $w_\mu$ we obtain that $w_\mu\in H^1_\mu$, and that $w_{\mu\lceil B_R}\rau w_\mu$ weakly in $H^1_\mu$ and that $w_\mu$ satisfies the equation 
\begin{equation}\label{eq02}
-\Delta w_\mu+w_\mu+\mu w_\mu=1
\end{equation}
in the weak sense, for test functions $v\in H^1_\mu$, with compact support.

Indeed, we have that $R\ra w_{\mu\lceil B_R}(x)$ is not decreasing and  and  $ w_{\mu\lceil B_R}(x)\ra w_\mu(x)$ a.e. But
$$\|w_{\mu\lceil B_R}\|_{H^1_\mu}^2 =\|w_{\mu\lceil B_R}\|_{L^1(\R^N)}.$$
The mapping $R \ra \|w_{\mu\lceil B_R}\|_{L^1(\R^N)}$ is not decreasing, and  for $R \ra +\infty$ we have $\|w_{\mu\lceil B_R}\|_{L^1(\R^N)} \ra \|w_{\mu}\|_{L^1(\R^N)}$ by the monotone convergence theorem. This proves that $(w_{\mu\lceil B_R})_R$ is bounded in $H^1_\mu$ and wealky converges in $H^1_\mu$ to $w_\mu$. Consequently, taking a test function $\vphi \in H^1_\mu$, with compact support, in the equation satisfied by $w_{\mu\lceil B_R}$ for $R$ large enough, we obtain by passage to the limit (\ref{eq02}).

Take now an arbitrary function  $v\in H^1_\mu$, $v\ge0$. We prove first that $v\in L^1(\R^N)$. We may take $\theta_Rv$ as test function in  (\ref{eq02}) and get
$$\int_{\R^N}\nabla w_\mu\nabla(\theta_R v)\,dx
+\int_{\R^N}w_\mu\theta_R v\,dx
+\int_{\R^N}w_\mu\theta_R v\,d\mu=\int_{\R^N}\theta_R v\,dx.$$
Passing to the limit as $R\to+\infty$, we obtain
$$\int_{\R^N}\nabla w_\mu\nabla v\,dx
+\int_{\R^N}w_\mu v\,dx+\int_{\R^N}w_\mu v\,d\mu=\int_{\R^N}v\,dx.$$
Since the left hand side is finite, then $v\in L^1(\R^N)$, and
$$\|v\|_{L^1(\R^N)}\le\|w_\mu\|_{H^1_\mu}\|v\|_{H^1_\mu}.$$
Therefore the embedding $H^1_\mu\subset L^1(\R^N)$ is continuous. This also means that $1\in(H^1_\mu)'$ and that equation \eqref{eq02} holds in $(H^1_\mu)'$.

In order to prove the compactness of the embedding, we assume that $v_n\in H^1_\mu$ has norm bounded by $M$ and converges weakly to $0$. Again, we may assume $v_n\ge0$. We consider $(1-\theta_R)v_n$ as test function for \eqref{eq02}.
Thus
\begin{equation}\label{neweq}
\begin{array}{ll}
&\ds\int_{\R^N}\nabla w_\mu\nabla\big((1-\theta_R)v_n\big)\,dx
+\int_{\R^N}w_\mu(1-\theta_R)v_n\,dx\\
&\ds\qquad+\int_{\R^N}w_\mu(1-\theta_R)v_n\,d\mu
=\int_{\R^N}(1-\theta_R)v_n\,dx.
\end{array}
\end{equation}
Since $w_\mu\in H^1_\mu$, for every $\vps>0$ there exists $R>0$ such that
$$\int_{\R^N\sm B_R}|\nabla w_\mu|^2\,dx
+\int_{\R^N\sm B_R}w_\mu^2\,dx
+\int_{\R^N\sm B_R}w_\mu^2\,d\mu\le\vps.$$
So, by \eqref{neweq} and H\"older inequality
$$\begin{array}{ll}
&\ds\int_{\R^N}(1-\theta_R)v_n\,dx
\le\vps^{1/2}\Big(\int_{\R^N}\big|\nabla\big((1-\theta_R)v_n\big)\big|^2\,dx
+\int_{\R^N}(1-\theta_R)v_n^2\,dx
+\int_{\R^N}(1-\theta_R)v_n^2\,d\mu\Big)^{1/2}\\
&\ds\qquad\le\vps^{1/2}\Big(2\int_{\R^N}|\nabla v_n|^2\,dx
+\frac{2|\nabla\theta|_\infty^2}{R^2}\int_{\R^N}v_n^2\,dx
+\int_{\R^N}v_n^2\,dx+\int_{\R^N}v_n^2\,d\mu\Big)^{1/2}.
\end{array}$$
Hence
$$\int_{\R^N}(1-\theta_R)v_n\,dx\le(2\vps)^{1/2}M\Big(1+\frac{|\nabla\theta|_\infty^2}{R^2}\Big)^{1/2}.$$
Since for every fixed $R$ we have
$$\lim_{\nif}\int_{\R^N}\theta_R v_n\,dx=0,$$
by a standard argument we get $|v_n|_{L^1(\R^N)}\to0$ as $\nif$.

\bigskip\noindent{\bf2) $\Rightarrow$ 1)} Let $C$ be the norm of the continuous injection $H^1_\mu\subset L^1(\R^N)$. Taking $w_{\mu\lceil B_R}$ as test function for $w_{\mu\lceil B_R}$ we get
$$\|w_{\mu\lceil B_R}\|^2_{H^1_\mu}
=\|w_{\mu\lceil B_R}\|_{L^1(\R^N)}\le C\|w_{\mu\lceil B_R}\|_{H^1_\mu}.$$
Consequently, $w_{\mu\lceil B_R}$ is uniformly bounded in $L^1(\R^N)$ and in $H^1_\mu$, for every $R>0$. Using the definition of $w_\mu$ and the monotone convergence theorem, we get that $w_\mu\in L^1(\R^N)$.
\end{proof}

Proposition \ref{p01} gives a list of useful tools for proving the compact embedding in $L^2(\R^N)$.

\begin{proof}[of Proposition \ref{p01}]

\noindent{\bf1) $\Rightarrow$ 5)}
Let $u_R \in H^1_{\mu\lceil B_R^c}$ be such that $\|u_R\|_{L^2 (\R^N)}=1$ and
$$\lb_1(\mu\lceil B_R^c)\le\frac{\int_{\R^N}|\nabla u_R|^2\,dx
+\int_{\R^N}u_R^2\,dx+\int_{\R^N}u_R^2\,d\mu}{\int_{\R^N}u_R^2\,dx}
\le1+\lb_1(\mu\lceil B_R^c).$$
Assume by contradiction that $\lb_1(\mu\lceil B_R^c)\not\to+\infty$. Then $(u_R)_R$ is bounded in $H^1_\mu$ and converges weakly to $0$ in $H^1_\mu$. This is a consequence of the fact that the support of $u_R$ is located outside the ball $B_R$. Condition 1) implies that $u_R$ has to converge strongly to $0$ in $L^2(\R^N)$ which is in contradiction with the fact that $\|u_R\|_{L^2 (\R^N)}=1$.

\bigskip\noindent{\bf5) $\Rightarrow$ 1)} The proof of this statement is implicitly contained in the proof of Theorem \ref{t01}.

\bigskip\noindent{\bf1) $\Rightarrow$ 3)} Assume 1) holds. In order to prove that $\Rc_{\mu\lceil B_R}\to\Rc_\mu$ in $\Lc(L^2(\R^N))$ as $R\to+\infty$ it is enough to consider a sequence $(f_n)_n$ in $L^2(\R^N)$ such that $f_n\rau f$ weakly in $L^2(\R^N)$ and prove that
$$\Rc_{\mu\lceil B_{R_n}}(f_n)\to\Rc_{\mu}(f)\mbox{ strongly in }L^2(\R^N).$$
For simplicity we set $\mu_n=\mu\lceil B_{R_n})$. Because of the compact injection assumption 1) and from the equiboundedness
$$\|\Rc_{\mu_n}(f_n)\|_{H^1_\mu}\le\|f_n\|_{L^2}$$
it is enough to prove simply that $\Rc_{\mu_n}(f_n)$ converges to $\Rc_{\mu}(f)$ weakly in $L^2(\R^N)$. Since $\Rc_{\mu_n}$ and $\Rc_{\mu}$ are self-adjoint operators this means
$$\langle f_n,\Rc_{\mu_n}(\psi)\rangle_{L^2}\to\langle f,\Rc_{\mu}(\psi)\rangle_{L^2}\qquad\forall\psi\in L^2(\R^N).$$
This will be a consequence of the fact that
\begin{equation}\label{weakR}
\Rc_{\mu_n}(\psi)\rau\Rc_{\mu}(\psi)\quad\hbox{ weakly in }L^2(\R^N),
\qquad\forall\psi\in L^2(\R^N)
\end{equation}
again because of the compact injection hypothesis 1). In order to prove \eqref{weakR} it is enough to assume $\psi\ge0$, so that the maximum principle gives that $\Rc_{\mu_n}(\psi)$ is a nondecreasing sequence of functions. If we denote $u_n=\Rc_{\mu_n}(\psi)$ then $u_n$ solves
\begin{equation}
\left\{\begin{array}{ll}
-\Delta u_n+u_n+u_n\mu_n=\psi\mbox{ in }(H^1_{\mu_n})'\\
u_n\in H^1_{\mu_n}.
\end{array}\right.
\end{equation}
There exists a constant $C$ such that for every $n$ we have $\|u_n\|_{H^1_\mu}\le C$, so without loss of generality we may assume that $u_n\rau u$ weakly in $H^1_\mu $, so by 1) strongly in $L^2(\R^N)$. Summarizing, in order to prove 3) it is enough to show that $u=\Rc_{\mu}(\psi)$, which is equivalent to
\begin{itemize}
\item $u\in H^1_\mu$: this is obvious;
\item $u$ solves the equation 
\begin{equation}\label{eq01}
\ds-\Delta u+u+u\mu=\psi\mbox{ in }(H^1_{\mu})'.
\end{equation}
\end{itemize}
We consider a test function $\vphi\in H^1_{\mu}$, with compact support. In order to prove that
$$\int_{\R^N}\nabla u\nabla\vphi\,dx+\int_{\R^N}u\vphi\,dx
+\int_{\R^N}u\vphi\,d\mu=\int_{\R^N}\psi\vphi\,dx,$$
we take $\vphi$ as test function for $u_n$, with $n$ large enough. Passing to the limit as $\nif$ we readily get \eqref{eq01}.

\bigskip\noindent{\bf3) $\Rightarrow$ 2)} This is an obvious consequence of the fact that $\Rc_{\mu\lceil B_R}$ are compact operators.

\bigskip\noindent{\bf 2) $\Rightarrow$ 5)} Assume by contradiction that $\lb_1(\mu\lceil B^c_R)\le M$ for every $R>0$. We denote $f_n=\lb_1(\mu\lceil B^c_n)u_n$, where $u_n\in H^1_{\mu\lceil B^c_n}$ is a first positive eigenfunction associated to the measure $\mu\lceil(B^c_n\cap B_{n'})$, with $\|u_n\|_{L^2}=1$, where $n'$ is large enough such that
$$\lb_1\big(\mu\lceil (B^c_n\cap B_{n'})\big)\le1+\lb_1(\mu\lceil B^c_n).$$
By monotonicty, we have that
$$u_n\le\Rc_{\mu}\big(\lb_1(\mu\lceil(B^c_n\cap B_{n'}))u_n\big),$$
so that $\Rc_{\mu}\big(\lb_1(\mu\lceil(B^c_n\cap B_{n'}))u_n\big)$ cannot converge strongly to $0$ in $L^2(\R^N)$. On the other hand, if we denote $f_n=\lb_1(\mu\lceil(B^c_n\cap B_{n'}))u_n$, we notice that $f_n$ converges to $0$ weakly in $L^2(\R^N)$, which contradicts hypothesis 2).

\bigskip\noindent{\bf4) $\Rightarrow$ 1)} Clearly, if 4) holds, then $\Rc_{\mu\lceil B_R}(1)\to\Rc_{\mu}(1)$ in $L^\infty(\R^N)$, hence
$$w_\mu-w_{\mu\lceil B_R}\to0\mbox{ in }L^\infty(\R^N),$$
so by monotonicity we have that $1_{B_R^c}\cdot w_\mu\to0$ in $L^\infty$, hence we can use Theorem \ref{t01}.

\bigskip\noindent{\bf1) $\Rightarrow$ 4)} In the case of a domain $\Om$, i.e. $\mu=\infty_{\R^N\sm\Om}$, this assertion is a consequence of Theorem \ref{t01} and \cite[Theorem 3.31]{bida06}. For measures, the proof is the same and is a direct consequence of Theorem \ref{t01} and of the inequality
$$0\le\Rc_\mu(f)-\Rc_{\mu\lceil B_n^c}(f)\le\|f\|_\infty\big(\Rc_\mu(1)-\Rc_{\mu\lceil B_n^c}(1)\big)\qquad\forall f\in L^\infty(\R^N),\ f\ge0.$$

\bigskip\noindent{\bf5) $\Rightarrow$ 6)} This is obvious, by monotonicity of the eigenvalues with respect to measures.

\bigskip\noindent{\bf6) $\Rightarrow$ 7)} This is obvious.

\bigskip\noindent{\bf7) $\Rightarrow$ 5)} Let $h$ be given by 7) and consider $h'<h$ such that a N-cube $C_{h'}$ of edge of length $h'$ centered at the origin is contained in the ball $B_h$. We cover the space $\R^N$ by closed cubes of edges of length $h'$ parallel to the axes and with centers in the points of the lattice $(\frac{h'}{2}\Z)^N$. We denote these cubes by $C_i$ for $i\in I$. We consider a function $\vphi\in C_c^\infty(C_{h'})$, such that $0\le\vphi\le1$, and $\vphi=1$ on the cube $C_{h'/2}$. We denote by $\vphi_i$ the function $\vphi$ supported by the cube $C_i$.

From hypothesis 7) for every $\vps>0$ there exists $R>0$ such that for all $\|x\|\ge R$
$$\lb_1(\mu\lceil B_h(x))\ge\frac{1}{\vps}.$$
Let us consider a function $u\in H^1_{\mu\lceil B_R^c}, u\ge 0$ such that 
\begin{equation}\label{ineq}
\lb_1(\mu\lceil B_R^c)\le\frac{\int_{\R^N}|\nabla u|^2\,dx
+\int_{\R^N}u^2\,dx+\int_{\R^N}u^2\,d\mu}
{\int_{\R^N}u^2\,dx}\le\lb_1(\mu\lceil B_R^c)+1.
\end{equation}
Then $\vphi_iu$ is a test function for $\lb_1(\mu\lceil C_i)$, and by monotonicity for $\lb_1(\mu\lceil B_i)$, where $B_i$ is the ball centered at the same point as $C_i$, with radius $h$. Consequently, we may write
$$\lb_1(\mu\lceil B_i)\int_{\R^N}|\vphi_i u|^2\,dx
\le\int_{\R^N}|\nabla(\vphi_i u)|^2\,dx
\int_{\R^N}|\vphi_i u|^2\,dx+\int_{\R^N}|\vphi_i u|^2\,d\mu.$$
Summing over $i\in I$ and decomposing $\nabla(\vphi_i u)$, we get
$$\begin{array}{ll}
&\ds\sum_i\lb_1(\mu\lceil B_i)\int_{\R^N}|\vphi_i u|^2\,dx
\le2\sum_i\int_{\R^N}|\vphi_i|^2|\nabla u|^2\,dx
+2\sum_i\int_{\R^N}u^2|\nabla\vphi_i|^2\,dx\\
&\ds\hskip3truecm+\sum_i\int_{\R^N}|\vphi_i u|^2\,dx
+\sum_i\int_{\R^N}|\vphi_i u|^2\,d\mu.
\end{array}$$
We notice that every point of the space is covered by at most $2^N$ cubes 
$C_i$ so we can write
$$\begin{array}{ll}
&\ds\frac{1}{\vps}\sum_i\int_{\R^N}|\vphi_i u|^2\,dx
\le2^{N+1}\Big(\int_{\R^N}|\nabla u|^2\,dx
+\|\nabla\vphi\|^2_\infty\int_{\R^N}u^2\,dx\Big)\\
&\ds\hskip3truecm+2^N\int_{\R^N}u^2\,dx+2^N\int_{\R^N}u^2\,d\mu.
\end{array}$$
For the left hand side, the average inequality gives
$$\sum_i\int_{\R^N}|\vphi_i u|^2\,dx
\ge\frac{1}{2^N}\int_{\R^N}u^2(\sum_i\vphi_i)^2\,dx
\ge\frac{1}{2^N}\int_{\R^N}u^2\,dx.$$
Consequently,
$$\begin{array}{ll}
&\ds\frac{1}{\vps}\frac{1}{2^N}\int_{\R^N}u^2\,dx
\le2^{N+1}\Big(\int_{\R^N}|\nabla u|^2\,dx
+|\nabla\vphi|^2_\infty\int_{\R^N}u^2\,dx\Big)\\
&\ds\hskip3truecm+2^N\int_{\R^N}u^2\,dx+2^N\int_{\R^N}u^2\,d\mu,
\end{array}$$
so that, using \eqref{ineq}
$$\frac{1}{\vps}\frac{1}{2^N}
\le2^{N+1}\big(\lb_1(\mu\lceil B_R^c)+\|\nabla\vphi\|^2_\infty\big)+2^N.$$
From this inequality, obviously 5) is a consequence of 7).

\bigskip\noindent{\bf8) $\Rightarrow$ 5)} Let us consider a function $u\in H^1_{\mu\lceil B_R^c}$, $u\ge0$, such that 
$$\lb_1(\mu\lceil B_R^c)\le\frac{\int_{\R^N}|\nabla u|^2\,dx
\int_{\R^N}u^2\,dx+\int_{\R^N}u^2\,d\mu}
{\int_{\R^N}u^2\,dx}\le\lb_1(\mu\lceil B_R^c)+1.$$
We consider the cubes $H_i=[x_i,x_i+h)$, where $x_i$ belongs to the lattice $(h\Z)^N$, and $\|x_i\|\ge r$. Then we can write
$$\int_{H_i}|\nabla u|^2\,dx
+\int_{H_i}u^2\,d\mu\ge\frac{1}{\vps}\int_{H_i}u^2\,dx,$$
and summing over $i\in I$ we obtain
$$\lb_1(\mu\lceil B_R^c)+1\ge\frac{1}{\vps}+1,$$
and 5) holds.

\bigskip\noindent{\bf7) $\Rightarrow$ 8)} Assume by contradiction that there exists $\vps>0$ such that on a sequence of cubes $H_n=[x_n,x_n+h)$, we have functions $u_n\in H^1_{\mu\lfloor H_n}$ with $\int_{H_n}u_n^2\,dx=1$ and
$$\int_{H_n}|\nabla u_n|^2\,dx+\int_{H_n}u_n^2\,d\mu
\le\frac{1}{\vps}\;.$$
Let $v_n=\min\{w_n,u_n\}$, where $w_n=w_{H_n}$. We get
$$\int_{H_n}|\nabla v_n|^2\,dx+\int_{H_n}v_n^2\,d\mu
\le\int_{H_n}|\nabla u_n|^2\,dx+\int_{H_n}|\nabla w_n|^2\,dx
+\int_{H_n}u_n^2\,d\mu\le\frac{1}{\vps}+|H_n|.$$
We shall prove that 7) fails, since $v_n\in H^1(\mu\lceil H_n)$ satisfies
$$\lb_1(\mu\lceil H_n)\int_{H_n}v_n^2\,dx
\le\int_{H_n}|\nabla v_n|^2\,dx+\int_{H_n}v_n^2\,dx+\int_{H_n}v_n^2\,d\mu
\le\frac{1}{\vps}+|H_n|+\int_{H_n}v_n^2\,d\mu.$$
If $\limsup_\nif\int_{H_n}v_n^2\,dx>0$, then 7) fails, in contradiction with the hypothesis.

Assume that $\lim_\nif\int_{H_n}v_n^2\,dx=0$. Making translations at the origin and passing to the limit, we get, by an abuse of notation,
$$\int_H(\min\{w_H, u\})^2\,dx=0,$$
where $u$ is a weak limit of the translations of $u_n$, hence has its $L^2$-norm equal to $1$. Obviously, this is impossible.
\end{proof}

\section {The general case}\label{secgene}

Let $1<p<+\infty$. We denote by $\MoNp$ the family of all nonnegative regular Borel measures, possibly $+\infty$ valued, that vanish on all sets of $p$-capacity zero (see the precise definition in \cite{hkm93}). For $\mu\in\MoNp$, we denote by $W^{1,p}_\mu$ the Banach space $W^{1,p}(\R^N)\cap L^p (\R^N,\mu)$, endowed with the norm
$$\|u\|^p_{1,\mu}=\int_{\R^N}\big(|\nabla u|^p+|u^p|\big)\,dx+\int_{\R^N}  |u|^p\,d\mu.$$
If $\mu\in\MoNp$, for every $f\in L^{p'}(\R^N)$ (where $\frac{1}{p}+\frac{1}{p'}=1$, or more generally for $f\in(W^{1,p}_\mu)'$) we may consider the elliptic PDE
\begin{equation}\label{pden}
-\Delta_p u+|u|^{p-2}u+\mu|u|^{p-2}u=f,\qquad u\in W^{1,p}_\mu
\end{equation}
whose precise sense has to be given in the weak form
$$\int_{\R^N}\big(|\nabla u|^{p-2}\nabla u\nabla\vphi+|u|^{p-2}u\vphi)\,dx
+\int |u|^{p-2}u\vphi\,d\mu
=\int_{\R^N}f\vphi \,dx \qquad\forall\phi\in W^{1,p}_\mu.$$

We still denote $\Rc(f)=u$, the nonlinear operator $\Rc:L^p_\mu\to L^p_\mu$ which associates to every $f$ the unique solution $u$ of the equation \eqref{pden}. Again, we note $w_\mu=\Rc(1)$, defined by approximation on the ball $B_R$, as in the linear case. We refer to \cite{dammu} for the study of monotonicity properties of $\Rc$ with respect to the measures $\mu$.

Then we have the following results.

\begin{thm}\label{t01p}
Let $\mu\in\MoNp$. The following assertions are equivalent.

\begin{enumerate}
\item $W^{1,p}_\mu\hr L^p(\R^N)$ is compact;
\item $w_\mu\cdot 1_{B_R^c}\to0$ in $L^\infty(\RR^N)$ as $R\to+\infty$;
\item $\Rc_\mu(f_n)\to\Rc_\mu(f)$ strongly in $L^p(\R^N)$, as soon as $f_n\rau f$ weakly in $L^{p'}(\R^N)$;
\item $\Rc_{\mu\lceil B_{R_n}}(f_n)\to\Rc_\mu(f)$ strongly in $L^p(\R^N)$, as soon as $f_n\rau f$ weakly in $L^{p'}(\R^N)$ and $R_n\to+\infty$;
\item $\lb_1(\mu\lceil B_R^c)\to+\infty$ as $R\to+\infty$;
\item for every $h>0$, $\lb_1(\mu\lceil B_h(x))\to+\infty$ as $\|x\|\to+\infty$;
\item there exists $h>0$ such that $\lb_1(\mu\lceil B_h(x))\to+\infty$ as $\|x\|\to+\infty$;
\item there exists $h>0$ such that for every $\vps>0$ there exists $r\ge0$ such that
$$\inf_{u\in W^{1,p}_{\mu\lfloor H}(\R^N),\ u\ne0}\frac{\int_H|\nabla u|^p\,dx
+\int_H u^p\,d\mu}{\int_Hu^p\,dx}\ge\frac{1}{\vps},$$
for every $N$-cube $H=[x,x+h)$, with $\|x\|\ge r$.
\end{enumerate}
\end{thm}

\begin{proof}
The only points different from those of Theroem \ref{t01} and Proposition \ref{p01} and which need some attention are listed below.

\bigskip\noindent{\bf1) $\Rightarrow$ 2)} The proof is similar to the necessity part of Theorem \ref{t01}. The only different point is concerned with the uniform bound of the $W^{1,p}_\mu$-norm of $\vphi_n$. We replace $\theta$ by another function $\beta\in C^\infty_c(B_2)$, with $0\le\beta\le1$, which will be precised later.

Notice first that 
$$\begin{array}{ll}
&\ds\int_{\R^N}|\nabla(w_{\mu\lceil B_{R_n}}\beta(\cdot+x_n))|^p\,dx
\le2^{p-1}\int_{\R^N}|\nabla w_{\mu\lceil B_{R_n}}|^p\beta(\cdot+x_n)^p\,dx\\
&\ds\hskip3truecm+2^{p-1}\int_{\R^N}w_{\mu\lceil B_{R_n}}^p|\nabla\beta(\cdot+x_n)|^p\,dx.
\end{array}$$
Taking $w_{\mu\lceil B_{R_n}}\beta(\cdot+x_n)$ as test function for the equation satisfied by $w_{\mu\lceil B_{R_n}}$, we get 
\begin{equation}\label{eq07}
\begin{array}{ll}
&\ds\int_{\R^N}|\nabla w_{\mu\lceil B_{R_n}}|^p\beta(\cdot+x_n)\,dx
+\int_{\R^N}w_{\mu\lceil B_{R_n}}|\nabla w_{\mu\lceil B_{R_n}}|^{p-2}\nabla w_{\mu\lceil B_{R_n}}\nabla\beta(\cdot+x_n)\,dx\\
&\ds\hskip2truecm+\int_{\R^N}w_{\mu\lceil B_{R_n}}^p\beta(\cdot+x_n)\,dx
+\int_{\R^N}w_{\mu\lceil B_{R_n}}^p\beta(\cdot+x_n)d\mu\le|B_2|.
\end{array}
\end{equation}
Setting
$$I=\int_{\R^N}w_{\mu\lceil B_{R_n}}|\nabla w_{\mu\lceil B_{R_n}}|^{p-2}\nabla
w_{\mu\lceil B_{R_n}}\nabla\beta(\cdot+x_n)\,dx,$$
by H\"older inequality we get
$$|I|\le\int_{\R^N}|\nabla w_{\mu\lceil B_{R_n}}|^{p-1}|\nabla\beta(\cdot+x_n)|\,dx
\le\Big(\int_{\R^N}|\nabla w_{\mu\lceil B_{R_n}}|^{p}|\nabla\beta(\cdot+x_n)|^\frac{p}{p-1} \,dx\Big)^{\frac{p-1}{p}}|B_2|^\frac{1}{p}.$$
We choose $\beta(x)= (1-|x|)^p\cdot1_{B_1}$ and notice that
$$|\nabla\beta|^\frac{p}{p-1}\le p^\frac{p}{p-1}\beta.$$
Thus, there exists a constant $C$ such that
\begin{equation}\label{eq06}
|I|\le C\Big(\int_{\R^N}|\nabla w_{\mu\lceil B_{R_n}}|^p\beta(\cdot+x_n) \,dx\Big)^\frac{p-1}{p}.
\end{equation}
Since $0\le\beta\le1$ we get the uniform bound of $\vphi_n$ by plugging \eqref{eq06} in \eqref{eq07}.

For a similar inequality to \eqref{eq03}, we may use \cite[Theorem 3.9]{mazi97}. 

\bigskip\noindent{\bf1) $\Lra$ 4)} One has only to prove that the weak $W^{1,p}_\mu$-limit of $\Rc_{\mu\lceil B_{R_n}}(f_n)$ is $\Rc_\mu (f)$. Assume that for some subsequence (still denoted with the same index) we have that $\Rc_{\mu\lceil B_{R_n}}(f_n)$ converges weakly in $W^{1,p}_\mu$ to some function $v$. We shall prove that $v=\Rc_\mu (f)$ relying on the $\G$-convergence principle. Indeed, on the one hand we have 
$$\int_{\R^N} |\nabla v|^p dx +\int_{\R^N} | v|^p dx+\int_{\R^N} | v|^p d\mu \leq \liminf \|\Rc_{\mu\lceil B_{R_n}}(f_n)\|_{W^{1,p}_\mu}^p.$$
If the inequality above is strict, using hypothesis 1), there exists $R$ large enough such that 
$$\frac{1}{p} (\int_{\R^N} |\nabla v|^p dx +\int_{\R^N} | v|^p dx+\int_{\R^N} | v|^p d\mu)-\int_{\R^N} fv dx $$
$$< \frac{1}{p} (\int_{\R^N} |\nabla (\theta_Rv)|^p dx +\int_{\R^N} |\theta_R v|^p dx+\int_{\R^N} |\theta_R v|^p d\mu)-\int_{\R^N} f\theta_Rv dx$$
$$ <  \liminf \frac{1}{p} \|\Rc_{\mu\lceil B_{R_n}}(f_n)\|_{W^{1,p}_\mu}^p -\int_{\R^N} f_n \Rc_{\mu\lceil B_{R_n}} dx.$$
Consequently, for some $n$ large enough, we also have
$$  \frac{1}{p} (\int_{\R^N} |\nabla (\theta_Rv)|^p dx +\int_{\R^N} |\theta_R v|^p dx+\int_{\R^N} |\theta_R v|^p d\mu)-\int_{\R^N} f_n\theta_Rv dx$$
$$ <   \frac{1}{p} \|\Rc_{\mu\lceil B_{R_n}}(f_n)\|_{W^{1,p}_\mu}^p -\int_{\R^N} f_n \Rc_{\mu\lceil B_{R_n}} dx,$$
which is in contradiction with the variational definition of $ \Rc_{\mu\lceil B_{R_n}}$.

Finally,  we get that $\Rc_{\mu\lceil B_{R_n}}(f_n)$ converges strongly in $W^{1,p}_\mu$ to $v$, thus one can pass to the limit the weak form of the equations associated to $ \Rc_{\mu\lceil B_{R_n}}$, for a fixed test function with compact support.
\end{proof}

\begin{thm}\label{t02p}
Let $\mu\in\MoNp$. The following assertions are equivalent.
\begin{enumerate}
\item $w_\mu\in L^1(\R^N)$;
\item $W^{1,p}_\mu\subset L^1(\R^N)$ and the injection is continuous.
\end{enumerate}
Moreover, if one of the two assertions above holds, then the embedding $W^{1,p}_\mu\hr L^1(\R^N)$ is compact.
\end{thm}

\begin{proof}
The proof follows the same steps as for Theorem \ref{t02}. The only point which is slightly different is to prove that $w_\mu$ satisfies the equation
\begin{equation}
-\Delta_p w_\mu+w_\mu^{p-2} w_\mu+\mu w_\mu^{p-2}w_\mu=1
\end{equation}
in the weak sense, for test functions $v\in W^{1,p}_\mu$ with compact support. The argument above for the proof of {\bf1) $\Lra$ 4)} can be repeated.
\end{proof}

\bibliographystyle{siam}

\def\cprime{$'$} \def\cprime{$'$}

\end{document}